\documentclass{amsart}
\usepackage{amssymb}
\usepackage{amsfonts}

\newtheorem{theorem}{Theorem}
\theoremstyle{plain}

\newtheorem{corollary}{Corollary}

\newtheorem{definition}{Definition}

\newtheorem{lemma}{Lemma}

\newtheorem{remark}{Remark}

\numberwithin{equation}{section}
\input{tcilatex}

\begin{document}
\title[Ostrowski's Type Inequalities]{Ostrowski's Type Inequalities for
Strongly$-$Convex Functions}
\author{$^{\bigstar }$Erhan Set}
\address{$^{\bigstar }$D\"{u}zce University, Faculty of Science and Arts,
Department of Mathematics, Konuralp Campus, D\"{u}zce, Turkey}
\email{erhanset@yahoo.com}
\author{$^{\blacktriangledown }$M. Emin \"{O}zdemir}
\address{$^{\blacktriangledown }$Ataturk University, K. K. Education
Faculty, Department of Mathematics, 25640, Kampus, Erzurum, Turkey}
\email{emos@atauni.edu.tr}
\author{$^{\bigstar }$M. Zeki Sar\i kaya}
\address{$^{\bigstar }$D\"{u}zce University, Faculty of Science and Arts,
Department of Mathematics, Konuralp Campus, D\"{u}zce, Turkey}
\email{sarikayamz@gmail.com}
\author{$^{\blacktriangledown }$Ahmet Ocak Akdemir}
\address{$^{\blacktriangledown }$A\u{g}r\i\ \.{I}brahim \c{C}e\c{c}en
University, Faculty of Science and Letters, Department of Mathematics,
04100, A\u{g}r\i , Turkey}
\email{ahmetakdemir@agri.edu.tr}

\begin{abstract}
In this paper, we establish Ostrowski's type inequalities for strongly$-$%
convex functions where $c>0$ by using some classical inequalities and
elemantery analysis. We also give some results for product of two strongly$-$%
convex functions.
\end{abstract}

\maketitle

\section{INTRODUCTION}

Let $f:I\subset \left[ 0,\infty \right] \rightarrow 
\mathbb{R}
$ be a differentiable mapping on $I^{\circ }$, the interior of the interval $%
I$, such that $f^{\prime }\in L\left[ a,b\right] $ where $a,b\in I$ with $%
a<b $. If $\left\vert f^{\prime }\left( x\right) \right\vert \leq M$, then
the following inequality holds (see \cite{8}).

\begin{equation}
\left\vert f(x)-\frac{1}{b-a}\int_{a}^{b}f(u)du\right\vert \leq \frac{M}{b-a}%
\left[ \frac{\left( x-a\right) ^{2}+\left( b-x\right) ^{2}}{2}\right]
\label{h.1.1}
\end{equation}

This inequality is well known in the literature as the \textit{Ostrowski
inequality.}\textbf{\ }For some results which generalize, improve and extend
the inequality (\ref{h.1.1}) see (\cite{1},\cite{9}) and the references
therein.

Let us recall some known definitions and results which we will use in this
paper. A function $f:I\rightarrow \mathbb{R}$, $I\subseteq \mathbb{R}$ is an
interval, is said to be a convex function on $I$ if 
\begin{equation}
f\left( tx+\left( 1-t\right) y\right) \leq tf\left( x\right) +\left(
1-t\right) f\left( y\right)  \label{convex}
\end{equation}%
holds for all $x,y\in I$ and $t\in \left[ 0,1\right] $. If the reversed
inequality in (\ref{convex}) holds, then $f$ is concave.

Definition of strongly$-$convex functions was given by Polyak in 1966 as
following:

\begin{definition}
(See \cite{2}) $f:I\rightarrow 
\mathbb{R}
$ is called strongly$-$convex with modulus $c>0,$ if 
\begin{equation*}
f\left( tx+\left( 1-t\right) y\right) \leq tf\left( x\right) +\left(
1-t\right) f\left( y\right) -ct\left( 1-t\right) \left( x-y\right) ^{2}
\end{equation*}%
for all $x,y\in I$ and $t\in \left( 0,1\right) .$
\end{definition}

Strongly convex functions have been introduced by Polyak in \cite{2} and
they play an important role in optimization theory and mathematical
economics. Various properties and applicatins of them can be found in the
literature see (\cite{2}-\cite{7}) and the references cited therein.

In \cite{1}, Alomari \textit{et al. }proved following result:

\begin{corollary}
Let $f:I\subset \lbrack 0,\infty )\rightarrow 
\mathbb{R}
$ be a differentiable mapping on $I^{0}$ such that $f^{\prime }\in L\left[
a,b\right] ,$ where $a,b\in I$ with $a<b.$ If $\left\vert f^{\prime
}\right\vert ^{q}$ is convex on $\left[ a,b\right] ,$ $p>1$ and $\left\vert
f^{\prime }\right\vert \leq M,$ then the following inequality holds;%
\begin{equation}
\left\vert f\left( x\right) -\frac{1}{b-a}\int\limits_{a}^{b}f\left(
u\right) du\right\vert \leq \frac{M}{b-a}\left[ \frac{\left( x-a\right)
^{2}+\left( b-x\right) ^{2}}{\left( p+1\right) ^{\frac{1}{p}}}\right]
\label{1.2}
\end{equation}%
for each $x\in \left[ a,b\right] .$
\end{corollary}

The main purpose of this paper is to prove some new Ostrowski-type
inequality for strongly$-$convex functions and to give new results under
some special conditions of our Theorems. We also establish several integral
inequalities which involving product of strongly$-$convex and convex
functions.

\section{MAIN\ RESULTS}

To prove our main results we need the following lemma (see \cite{1}):

\begin{lemma}
Let $f:I\subset 
\mathbb{R}
\rightarrow 
\mathbb{R}
$ be a differentiable mapping on $I^{0}$ where $a,b\in I$ with $a<b.$ If $%
f^{\prime }\in L\left[ a,b\right] ,$ then the following equality holds;%
\begin{equation*}
f\left( x\right) -\frac{1}{b-a}\int\limits_{a}^{b}f\left( u\right) du=\frac{%
\left( x-a\right) ^{2}}{b-a}\int\limits_{0}^{1}tf^{\prime }\left( tx+\left(
1-t\right) a\right) dt-\frac{\left( b-x\right) ^{2}}{b-a}\int%
\limits_{0}^{1}tf^{\prime }\left( tx+\left( 1-t\right) b\right) dt
\end{equation*}%
for each $x\in \left[ a,b\right] .$
\end{lemma}

\begin{theorem}
Let $f:I\subset 
\mathbb{R}
\rightarrow 
\mathbb{R}
$ be a differentiable mapping on $I^{0}$ such that $f^{\prime }\in L\left[
a,b\right] ,$ where $a,b\in I$ with $a<b.$ If $\left\vert f^{\prime
}\right\vert $ is strongly$-$convex on $\left[ a,b\right] $ with respect to $%
c>0$, $\left\vert f^{\prime }\right\vert \leq M$ and $M\geq \max \left\{ 
\frac{c\left( x-a\right) ^{2}}{6},\frac{c\left( b-x\right) ^{2}}{6}\right\} ,
$ then the following inequality holds;%
\begin{eqnarray}
\left\vert f\left( x\right) -\frac{1}{b-a}\int\limits_{a}^{b}f\left(
u\right) du\right\vert  &\leq &\frac{\left( x-a\right) ^{2}}{2\left(
b-a\right) }\left( M-\frac{c\left( x-a\right) ^{2}}{6}\right)   \label{aa} \\
&&+\frac{\left( b-x\right) ^{2}}{2\left( b-a\right) }\left( M-\frac{c\left(
b-x\right) ^{2}}{6}\right) .  \notag
\end{eqnarray}%
for all $x,y\in \left[ a,b\right] \ $\ and $t\in \left( 0,1\right) .$
\end{theorem}

\begin{proof}
From Lemma 1 and by using the property of modulus$,$ we have%
\begin{eqnarray*}
\left\vert f\left( x\right) -\frac{1}{b-a}\int\limits_{a}^{b}f\left(
u\right) du\right\vert &\leq &\frac{\left( x-a\right) ^{2}}{b-a}%
\int\limits_{0}^{1}t\left\vert f^{\prime }\left( tx+\left( 1-t\right)
a\right) \right\vert dt \\
&&+\frac{\left( b-x\right) ^{2}}{b-a}\int\limits_{0}^{1}t\left\vert
f^{\prime }\left( tx+\left( 1-t\right) b\right) \right\vert dt.
\end{eqnarray*}%
Since $\left\vert f^{\prime }\right\vert $ is strongly$-$convex on $\left[
a,b\right] $ and $\left\vert f^{\prime }\right\vert \leq M,$ we get%
\begin{eqnarray*}
\int\limits_{0}^{1}t\left\vert f^{\prime }\left( tx+\left( 1-t\right)
a\right) \right\vert dt &\leq &\int\limits_{0}^{1}\left[ t^{2}\left\vert
f^{\prime }\left( x\right) \right\vert +t\left( 1-t\right) \left\vert
f^{\prime }\left( a\right) \right\vert -ct^{2}\left( 1-t\right) \left(
x-a\right) ^{2}\right] dt \\
&\leq &\frac{M}{2}-\frac{c\left( x-a\right) ^{2}}{12}
\end{eqnarray*}%
and%
\begin{eqnarray*}
\int\limits_{0}^{1}t\left\vert f^{\prime }\left( tx+\left( 1-t\right)
b\right) \right\vert dt &\leq &\int\limits_{0}^{1}\left[ t^{2}\left\vert
f^{\prime }\left( x\right) \right\vert +t\left( 1-t\right) \left\vert
f^{\prime }\left( b\right) \right\vert -ct^{2}\left( 1-t\right) \left(
b-x\right) ^{2}\right] dt \\
&\leq &\frac{M}{2}-\frac{c\left( b-x\right) ^{2}}{12}.
\end{eqnarray*}%
We can easily deduce%
\begin{eqnarray*}
\left\vert f\left( x\right) -\frac{1}{b-a}\int\limits_{a}^{b}f\left(
u\right) du\right\vert &\leq &\frac{\left( x-a\right) ^{2}}{2\left(
b-a\right) }\left( M-\frac{c\left( x-a\right) ^{2}}{6}\right) \\
&&+\frac{\left( b-x\right) ^{2}}{2\left( b-a\right) }\left( M-\frac{c\left(
b-x\right) ^{2}}{6}\right) .
\end{eqnarray*}%
which completes the proof.
\end{proof}

\begin{remark}
If we take $c\rightarrow 0^{+}$ in the inequality (\ref{aa}), we obtain the
inequality (\ref{h.1.1}).
\end{remark}

\begin{corollary}
If we choose $x=\frac{a+b}{2}$ in the inequality (\ref{aa}), we obtain the
following inequality:%
\begin{equation*}
\left\vert f\left( \frac{a+b}{2}\right) -\frac{1}{b-a}\int\limits_{a}^{b}f%
\left( u\right) du\right\vert \leq M\frac{\left( b-a\right) }{4}-\frac{%
c\left( b-a\right) ^{3}}{96}.
\end{equation*}
\end{corollary}

\begin{theorem}
Let $f:I\subset 
\mathbb{R}
\rightarrow 
\mathbb{R}
$ be a differentiable mapping on $I^{0}$ such that $f^{\prime }\in L\left[
a,b\right] ,$ where $a,b\in I$ with $a<b.$ If $\left\vert f^{\prime
}\right\vert ^{q}$ is strongly$-$convex on $\left[ a,b\right] $ with respect
to $c>0$, $\left\vert f^{\prime }\right\vert \leq M$ and $M^{q}\geq \max
\left\{ \frac{c\left( x-a\right) ^{2}}{6},\frac{c\left( b-x\right) ^{2}}{6}%
\right\} $ then the following inequality holds;%
\begin{eqnarray}
\left\vert f\left( x\right) -\frac{1}{b-a}\int\limits_{a}^{b}f\left(
u\right) du\right\vert  &\leq &\frac{\left( x-a\right) ^{2}}{b-a}\left( 
\frac{1}{p+1}\right) ^{\frac{1}{p}}\left( M^{q}-\frac{c\left( x-a\right) ^{2}%
}{6}\right) ^{\frac{1}{q}}  \label{a} \\
&&+\frac{\left( b-x\right) ^{2}}{b-a}\left( \frac{1}{p+1}\right) ^{\frac{1}{p%
}}\left( M^{q}-\frac{c\left( b-x\right) ^{2}}{6}\right) ^{\frac{1}{q}} 
\notag
\end{eqnarray}%
for all $x,y\in \left[ a,b\right] ,$ $t\in \left( 0,1\right) ,$ $q>1$ and $%
\frac{1}{p}+\frac{1}{q}=1.$
\end{theorem}

\begin{proof}
From Lemma 1 and by using the H\"{o}lder's inequality for $q>1,$ we have%
\begin{eqnarray*}
\left\vert f\left( x\right) -\frac{1}{b-a}\int\limits_{a}^{b}f\left(
u\right) du\right\vert &\leq &\frac{\left( x-a\right) ^{2}}{b-a}\left(
\int\limits_{0}^{1}t^{p}dt\right) ^{\frac{1}{p}}\left(
\int\limits_{0}^{1}\left\vert f^{\prime }\left( tx+\left( 1-t\right)
a\right) \right\vert ^{q}dt\right) ^{\frac{1}{q}} \\
&&+\frac{\left( b-x\right) ^{2}}{b-a}\left(
\int\limits_{0}^{1}t^{p}dt\right) ^{\frac{1}{p}}\left(
\int\limits_{0}^{1}\left\vert f^{\prime }\left( tx+\left( 1-t\right)
b\right) \right\vert ^{q}dt\right) ^{\frac{1}{q}}.
\end{eqnarray*}%
Since $\left\vert f^{\prime }\right\vert ^{q}$ is strongly$-$convex on $%
\left[ a,b\right] $ and $\left\vert f^{\prime }\right\vert ^{q}\leq M,$ we
get%
\begin{eqnarray*}
\int\limits_{0}^{1}\left\vert f^{\prime }\left( tx+\left( 1-t\right)
a\right) \right\vert ^{q}dt &\leq &\int\limits_{0}^{1}\left[ t\left\vert
f^{\prime }\left( x\right) \right\vert ^{q}+\left( 1-t\right) \left\vert
f^{\prime }\left( a\right) \right\vert ^{q}-ct\left( 1-t\right) \left(
x-a\right) ^{2}\right] dt \\
&\leq &M^{q}-\frac{c\left( x-a\right) ^{2}}{6}
\end{eqnarray*}%
and%
\begin{eqnarray*}
\int\limits_{0}^{1}\left\vert f^{\prime }\left( tx+\left( 1-t\right)
b\right) \right\vert ^{q}dt &\leq &\int\limits_{0}^{1}\left[ t\left\vert
f^{\prime }\left( x\right) \right\vert ^{q}+\left( 1-t\right) \left\vert
f^{\prime }\left( b\right) \right\vert ^{q}-ct\left( 1-t\right) \left(
b-x\right) ^{2}\right] dt \\
&\leq &M^{q}-\frac{c\left( b-x\right) ^{2}}{6}.
\end{eqnarray*}%
Therefore, we obtain%
\begin{eqnarray*}
\left\vert f\left( x\right) -\frac{1}{b-a}\int\limits_{a}^{b}f\left(
u\right) du\right\vert &\leq &\frac{\left( x-a\right) ^{2}}{b-a}\left( \frac{%
1}{p+1}\right) ^{\frac{1}{p}}\left( M^{q}-\frac{c\left( x-a\right) ^{2}}{6}%
\right) ^{\frac{1}{q}} \\
&&+\frac{\left( b-x\right) ^{2}}{b-a}\left( \frac{1}{p+1}\right) ^{\frac{1}{p%
}}\left( M^{q}-\frac{c\left( b-x\right) ^{2}}{6}\right) ^{\frac{1}{q}}.
\end{eqnarray*}%
which completes the proof.
\end{proof}

\begin{remark}
If we take $c\rightarrow 0^{+}$ in the inequality (\ref{a}), we obtain the
inequality (\ref{1.2}).
\end{remark}

\begin{corollary}
If we choose $x=\frac{a+b}{2}$ in the inequality (\ref{a}), we obtain the
following inequality:%
\begin{equation*}
\left\vert f\left( \frac{a+b}{2}\right) -\frac{1}{b-a}\int\limits_{a}^{b}f%
\left( u\right) du\right\vert \leq \frac{\left( b-a\right) }{2}\left( \frac{1%
}{p+1}\right) ^{\frac{1}{p}}\left( M^{q}-\frac{c\left( b-a\right) ^{2}}{24}%
\right) ^{\frac{1}{q}}.
\end{equation*}
\end{corollary}

\begin{theorem}
Let $f:I\subset 
\mathbb{R}
\rightarrow 
\mathbb{R}
$ be a differentiable mapping on $I^{0}$ such that $f^{\prime }\in L\left[
a,b\right] ,$ where $a,b\in I$ with $a<b.$ If $\left\vert f^{\prime
}\right\vert ^{q}$ is strongly$-$convex on $\left[ a,b\right] $ with respect
to $b,c>0,$ $q\geq 1$, $\left\vert f^{\prime }\right\vert \leq M$ and $%
M^{q}\geq \max \left\{ \frac{c\left( x-a\right) ^{2}}{6},\frac{c\left(
b-x\right) ^{2}}{6}\right\} $  then the following inequality holds;%
\begin{eqnarray}
\left\vert f\left( x\right) -\frac{1}{b-a}\int\limits_{a}^{b}f\left(
u\right) du\right\vert  &\leq &\frac{\left( x-a\right) ^{2}}{2\left(
b-a\right) }\left( M^{q}-\frac{c\left( x-a\right) ^{2}}{6}\right) ^{\frac{1}{%
q}}  \label{k} \\
&&+\frac{\left( b-x\right) ^{2}}{2\left( b-a\right) }\left( M^{q}-\frac{%
c\left( b-x\right) ^{2}}{6}\right) ^{\frac{1}{q}}  \notag
\end{eqnarray}%
for all $x,y\in \left[ a,b\right] \ $\ and $t\in \left( 0,1\right) .$
\end{theorem}

\begin{proof}
From Lemma 1 and applying the Power mean inequality for $q\geq 1,$ we have%
\begin{eqnarray*}
\left\vert f\left( x\right) -\frac{1}{b-a}\int\limits_{a}^{b}f\left(
u\right) du\right\vert &\leq &\frac{\left( x-a\right) ^{2}}{b-a}\left(
\int\limits_{0}^{1}tdt\right) ^{1-\frac{1}{q}}\left(
\int\limits_{0}^{1}t\left\vert f^{\prime }\left( tx+\left( 1-t\right)
a\right) \right\vert ^{q}dt\right) ^{\frac{1}{q}} \\
&&+\frac{\left( b-x\right) ^{2}}{b-a}\left( \int\limits_{0}^{1}tdt\right)
^{1-\frac{1}{q}}\left( \int\limits_{0}^{1}t\left\vert f^{\prime }\left(
tx+\left( 1-t\right) b\right) \right\vert ^{q}dt\right) ^{\frac{1}{q}}.
\end{eqnarray*}%
Since $\left\vert f^{\prime }\right\vert ^{q}$ is strongly$-$convex on $%
\left[ a,b\right] $ and $\left\vert f^{\prime }\right\vert ^{q}\leq M,$ we
get%
\begin{eqnarray*}
\int\limits_{0}^{1}t\left\vert f^{\prime }\left( tx+\left( 1-t\right)
a\right) \right\vert ^{q}dt &\leq &\int\limits_{0}^{1}\left[ t^{2}\left\vert
f^{\prime }\left( x\right) \right\vert ^{q}+t\left( 1-t\right) \left\vert
f^{\prime }\left( a\right) \right\vert ^{q}-ct^{2}\left( 1-t\right) \left(
x-a\right) ^{2}\right] dt \\
&\leq &\frac{M^{q}}{2}-\frac{c\left( x-a\right) ^{2}}{12}
\end{eqnarray*}%
and%
\begin{eqnarray*}
\int\limits_{0}^{1}t\left\vert f^{\prime }\left( tx+\left( 1-t\right)
b\right) \right\vert ^{q}dt &\leq &\int\limits_{0}^{1}\left[ t^{2}\left\vert
f^{\prime }\left( x\right) \right\vert ^{q}+t\left( 1-t\right) \left\vert
f^{\prime }\left( b\right) \right\vert ^{q}-ct^{2}\left( 1-t\right) \left(
b-x\right) ^{2}\right] dt \\
&\leq &\frac{M^{q}}{2}-\frac{c\left( b-x\right) ^{2}}{12}.
\end{eqnarray*}%
Hence, we deduce%
\begin{eqnarray*}
\left\vert f\left( x\right) -\frac{1}{b-a}\int\limits_{a}^{b}f\left(
u\right) du\right\vert &\leq &\frac{\left( x-a\right) ^{2}}{2\left(
b-a\right) }\left( M^{q}-\frac{c\left( x-a\right) ^{2}}{6}\right) ^{\frac{1}{%
q}} \\
&&+\frac{\left( b-x\right) ^{2}}{2\left( b-a\right) }\left( M^{q}-\frac{%
c\left( b-x\right) ^{2}}{6}\right) ^{\frac{1}{q}}.
\end{eqnarray*}%
which completes the proof.
\end{proof}

\begin{remark}
If we take $c\rightarrow 0^{+}$ in the inequality (\ref{k}), we obtain the
inequality (\ref{h.1.1}).
\end{remark}

\begin{corollary}
If we choose $x=\frac{a+b}{2}$ in the inequality (\ref{k}), we obtain the
following inequality:%
\begin{equation*}
\left\vert f\left( \frac{a+b}{2}\right) -\frac{1}{b-a}\int\limits_{a}^{b}f%
\left( u\right) du\right\vert \leq \frac{\left( b-a\right) }{4}\left( M^{q}-%
\frac{c\left( b-a\right) ^{2}}{24}\right) ^{\frac{1}{q}}.
\end{equation*}
\end{corollary}

\begin{theorem}
Suppose that $f,g:I\subset 
\mathbb{R}
\rightarrow \left[ 0,\infty \right) $ are strongly$-$convex functions on $%
I^{0}$ with respect to $c>0$ such that $fg\in L\left[ a,b\right] ,$ where $%
a,b\in I$ with $a<b.$Then the following inequality holds:%
\begin{eqnarray}
&&\frac{1}{b-a}\int\limits_{a}^{b}f\left( x\right) g\left( x\right) dx
\label{z1} \\
&\leq &\frac{1}{3}\left[ f\left( a\right) g\left( a\right) +f\left( b\right)
g\left( b\right) \right] +\frac{1}{6}\left[ f\left( a\right) g\left(
b\right) +f\left( b\right) g\left( a\right) \right]  \notag \\
&&-\frac{c\left( b-a\right) ^{2}}{12}\left[ f\left( a\right) +f\left(
b\right) +g\left( a\right) +g\left( b\right) \right] +\frac{c\left(
b-a\right) ^{4}}{30}.  \notag
\end{eqnarray}
\end{theorem}

\begin{proof}
From strongly$-$convexity of $f$ and $g$, we can write%
\begin{equation*}
f\left( tb+\left( 1-t\right) a\right) \leq tf\left( b\right) +\left(
1-t\right) f\left( a\right) -ct\left( 1-t\right) \left( b-a\right) ^{2}
\end{equation*}%
and%
\begin{equation*}
g\left( tb+\left( 1-t\right) a\right) \leq tg\left( b\right) +\left(
1-t\right) g\left( a\right) -ct\left( 1-t\right) \left( b-a\right) ^{2}
\end{equation*}%
Since $f,g$ are non-negative, we have%
\begin{eqnarray}
&&f\left( tb+\left( 1-t\right) a\right) g\left( tb+\left( 1-t\right) a\right)
\label{1} \\
&\leq &\left[ tf\left( b\right) +\left( 1-t\right) f\left( a\right)
-ct\left( 1-t\right) \left( b-a\right) ^{2}\right]  \notag \\
&&\times \left[ tg\left( b\right) +\left( 1-t\right) g\left( a\right)
-ct\left( 1-t\right) \left( b-a\right) ^{2}\right] .  \notag
\end{eqnarray}%
By integrating the resulting inequality with respect to $t$ over $\left[ 0,1%
\right] ,$ we get%
\begin{eqnarray*}
&&\int\limits_{0}^{1}f\left( tb+\left( 1-t\right) a\right) g\left( tb+\left(
1-t\right) a\right) dt \\
&\leq &\frac{1}{3}\left[ f\left( a\right) g\left( a\right) +f\left( b\right)
g\left( b\right) \right] +\frac{1}{6}\left[ f\left( a\right) g\left(
b\right) +f\left( b\right) g\left( a\right) \right] \\
&&-\frac{c\left( b-a\right) ^{2}}{12}\left[ f\left( a\right) +f\left(
b\right) +g\left( a\right) +g\left( b\right) \right] +\frac{c\left(
b-a\right) ^{4}}{30}.
\end{eqnarray*}%
Hence, by taking into account the change of the variable $tb+\left(
1-t\right) a=x,$ $(b-a)dt=dx,$ we obtain the required result.
\end{proof}

\begin{corollary}
If we choose $g\left( x\right) =1$ in (\ref{z1}), we obtain the following
inequality:%
\begin{equation*}
\frac{1}{b-a}\int\limits_{a}^{b}f\left( x\right) dx\leq \frac{f\left(
a\right) +f\left( b\right) }{2}-\frac{c\left( b-a\right) ^{2}}{12}\left[
f\left( a\right) +f\left( b\right) +2\right] +\frac{c\left( b-a\right) ^{4}}{%
30}.
\end{equation*}
\end{corollary}

\begin{theorem}
Suppose that $f,g:I\subset 
\mathbb{R}
\rightarrow \left[ 0,\infty \right) $ are strongly$-$convex functions on $%
I^{0}$ with respect to $c>0$ such that $fg\in L\left[ a,b\right] ,$ where $%
a,b\in I$ with $a<b.$Then the following inequality holds:%
\begin{eqnarray}
&&\frac{g\left( b\right) }{\left( b-a\right) ^{2}}\int\limits_{a}^{b}(x-a)f%
\left( x\right) dx+\frac{g\left( a\right) }{\left( b-a\right) ^{2}}%
\int\limits_{a}^{b}(b-x)f\left( x\right) dx  \label{z2} \\
&&+\frac{f\left( b\right) }{\left( b-a\right) ^{2}}\int\limits_{a}^{b}(x-a)g%
\left( x\right) dx+\frac{f\left( a\right) }{\left( b-a\right) ^{2}}%
\int\limits_{a}^{b}(b-x)g\left( x\right) dx  \notag \\
&&-\frac{c}{\left( b-a\right) ^{3}}\int\limits_{a}^{b}(x-a)(b-x)f\left(
x\right) dx-\frac{c}{\left( b-a\right) ^{3}}\int\limits_{a}^{b}(x-a)(b-x)g%
\left( x\right) dx  \notag \\
&\leq &\frac{1}{b-a}\int\limits_{a}^{b}f\left( x\right) g\left( x\right) dx+%
\frac{1}{3}\left[ f\left( a\right) g\left( a\right) +f\left( b\right)
g\left( b\right) \right] +\frac{1}{6}\left[ f\left( a\right) g\left(
b\right) +f\left( b\right) g\left( a\right) \right]  \notag \\
&&-\frac{c\left( b-a\right) ^{2}}{12}\left[ f\left( a\right) +f\left(
b\right) +g\left( a\right) +g\left( b\right) \right] +\frac{c\left(
b-a\right) ^{4}}{30}.  \notag
\end{eqnarray}
\end{theorem}

\begin{proof}
Since $f$ and $g$ are strongly$-$convex functions, we can write%
\begin{equation*}
f\left( tb+\left( 1-t\right) a\right) \leq tf\left( b\right) +\left(
1-t\right) f\left( a\right) -ct\left( 1-t\right) \left( b-a\right) ^{2}
\end{equation*}%
and%
\begin{equation*}
g\left( tb+\left( 1-t\right) a\right) \leq tg\left( b\right) +\left(
1-t\right) g\left( a\right) -ct\left( 1-t\right) \left( b-a\right) ^{2}
\end{equation*}%
By using the elementary inequality, $e\leq f$ and $p\leq r$, then $er+fp\leq
ep+fr$ for $e,f,p,r\in 
\mathbb{R}
,$ then we get%
\begin{eqnarray*}
&&f\left( tb+\left( 1-t\right) a\right) \left[ tg\left( b\right) +\left(
1-t\right) g\left( a\right) -ct\left( 1-t\right) \left( b-a\right) ^{2}%
\right] \\
&&+g\left( tb+\left( 1-t\right) a\right) \left[ tf\left( b\right) +\left(
1-t\right) f\left( a\right) -ct\left( 1-t\right) \left( b-a\right) ^{2}%
\right] \\
&\leq &f\left( tb+\left( 1-t\right) a\right) g\left( tb+\left( 1-t\right)
a\right) \\
&&+\left[ tf\left( b\right) +\left( 1-t\right) f\left( a\right) -ct\left(
1-t\right) \left( b-a\right) ^{2}\right] \left[ tg\left( b\right) +\left(
1-t\right) g\left( a\right) -ct\left( 1-t\right) \left( b-a\right) ^{2}%
\right] .
\end{eqnarray*}%
So, we obtain%
\begin{eqnarray*}
&&tf\left( tb+\left( 1-t\right) a\right) g\left( b\right) +\left( 1-t\right)
f\left( tb+\left( 1-t\right) a\right) g\left( a\right) -ct\left( 1-t\right)
f\left( tb+\left( 1-t\right) a\right) \left( b-a\right) ^{2} \\
&&+tf\left( b\right) g\left( tb+\left( 1-t\right) a\right) +\left(
1-t\right) f\left( a\right) g\left( tb+\left( 1-t\right) a\right) -ct\left(
1-t\right) g\left( tb+\left( 1-t\right) a\right) \left( b-a\right) ^{2} \\
&\leq &f\left( tb+\left( 1-t\right) a\right) g\left( tb+\left( 1-t\right)
a\right) +t^{2}f\left( b\right) g\left( b\right) \\
&&+t\left( 1-t\right) f\left( b\right) g\left( a\right) +t\left( 1-t\right)
f\left( a\right) g\left( b\right) +\left( 1-t\right) ^{2}f\left( a\right)
g\left( a\right) \\
&&-ct\left( 1-t\right) \left( b-a\right) ^{2}\left( t\left[ f\left( b\right)
+g\left( b\right) \right] +\left( 1-t\right) \left[ f\left( a\right)
+g\left( a\right) \right] \right) -ct^{2}\left( 1-t\right) ^{2}\left(
b-a\right) ^{4}.
\end{eqnarray*}%
By integrating this inequality with respect to $t$ over $\left[ 0,1\right] $
and by using the change of the variable $tb+\left( 1-t\right) a=x,$ $%
(b-a)dt=dx,$ the proof is completed.
\end{proof}

\begin{theorem}
Suppose that $f,g:I\subset 
\mathbb{R}
\rightarrow \left[ 0,\infty \right) $ are convex and strongly$-$convex
functions, respectively, on $I^{0}$ with respect to $c>0$ such that $fg\in L%
\left[ a,b\right] ,$ where $a,b\in I$ with $a<b.$Then the following
inequality holds:%
\begin{eqnarray}
&&\frac{1}{b-a}\int\limits_{a}^{b}f\left( x\right) g\left( x\right) dx+\frac{%
c\left( b-a\right) ^{2}}{6}\left[ \frac{f\left( a\right) +f\left( b\right) }{%
2}\right]  \label{z3} \\
&\leq &\frac{1}{3}\left[ f\left( a\right) g\left( a\right) +f\left( b\right)
g\left( b\right) \right] +\frac{1}{6}\left[ f\left( a\right) g\left(
b\right) +f\left( b\right) g\left( a\right) \right] .  \notag
\end{eqnarray}
\end{theorem}

\begin{proof}
Since $f$ is convex and $g$ is strongly$-$convex function, we can write%
\begin{equation*}
f\left( tb+\left( 1-t\right) a\right) \leq tf\left( b\right) +\left(
1-t\right) f\left( a\right)
\end{equation*}%
and%
\begin{equation*}
g\left( tb+\left( 1-t\right) a\right) \leq tg\left( b\right) +\left(
1-t\right) g\left( a\right) -ct\left( 1-t\right) \left( b-a\right) ^{2}
\end{equation*}%
By multiplying the above inequalities side by side, we have%
\begin{eqnarray*}
&&f\left( tb+\left( 1-t\right) a\right) g\left( tb+\left( 1-t\right) a\right)
\\
&\leq &\left[ tf\left( b\right) +\left( 1-t\right) f\left( a\right) \right] %
\left[ tg\left( b\right) +\left( 1-t\right) g\left( a\right) -ct\left(
1-t\right) \left( b-a\right) ^{2}\right] .
\end{eqnarray*}%
By integrating the resulting inequality with respect to $t$ over $\left[ 0,1%
\right] ,$ we get%
\begin{eqnarray*}
&&\int\limits_{0}^{1}f\left( tb+\left( 1-t\right) a\right) g\left( tb+\left(
1-t\right) a\right) dt \\
&\leq &\frac{1}{3}\left[ f\left( a\right) g\left( a\right) +f\left( b\right)
g\left( b\right) \right] +\frac{1}{6}\left[ f\left( a\right) g\left(
b\right) +f\left( b\right) g\left( a\right) \right] \\
&&-\frac{c\left( b-a\right) ^{2}}{6}\left[ \frac{f\left( a\right) +f\left(
b\right) }{2}\right] .
\end{eqnarray*}%
Hence, by taking into account the change of the variable $tb+\left(
1-t\right) a=x,$ $(b-a)dt=dx,$ we obtain the required result.
\end{proof}

\begin{corollary}
If we choose $g\left( x\right) =1$ in (\ref{z3}), we obtain the following
inequality:%
\begin{equation*}
\frac{1}{b-a}\int\limits_{a}^{b}f\left( x\right) dx\leq \left[ 1-\frac{%
c\left( b-a\right) ^{2}}{6}\right] \frac{f\left( a\right) +f\left( b\right) 
}{2}.
\end{equation*}
\end{corollary}

\begin{theorem}
Suppose that $f,g:I\subset 
\mathbb{R}
\rightarrow \left[ 0,\infty \right) $ are convex and strongly$-$convex
functions, respectively, on $I^{0}$ with respect to $c>0$ such that $fg\in L%
\left[ a,b\right] ,$ where $a,b\in I$ with $a<b.$Then the following
inequality holds:%
\begin{eqnarray*}
&&\frac{g\left( b\right) }{\left( b-a\right) ^{2}}\int\limits_{a}^{b}(x-a)f%
\left( x\right) dx+\frac{g\left( a\right) }{\left( b-a\right) ^{2}}%
\int\limits_{a}^{b}(b-x)f\left( x\right) dx \\
&&+\frac{f\left( b\right) }{\left( b-a\right) ^{2}}\int\limits_{a}^{b}(x-a)g%
\left( x\right) dx+\frac{f\left( a\right) }{\left( b-a\right) ^{2}}%
\int\limits_{a}^{b}(b-x)g\left( x\right) dx \\
&&-\frac{c}{\left( b-a\right) ^{3}}\int\limits_{a}^{b}(x-a)(b-x)f\left(
x\right) dx \\
&\leq &\frac{1}{b-a}\int\limits_{a}^{b}f\left( x\right) g\left( x\right) dx+%
\frac{1}{3}\left[ f\left( a\right) g\left( a\right) +f\left( b\right)
g\left( b\right) \right] \\
&&+\frac{1}{6}\left[ f\left( a\right) g\left( b\right) +f\left( b\right)
g\left( a\right) \right] -\frac{c\left( b-a\right) ^{2}}{6}\left[ \frac{%
f\left( a\right) +f\left( b\right) }{2}\right] .
\end{eqnarray*}
\end{theorem}

\begin{proof}
Since $f$ and $g$ are convex and strongly$-$convex functions, respectively,
we can write%
\begin{equation*}
f\left( tb+\left( 1-t\right) a\right) \leq tf\left( b\right) +\left(
1-t\right) f\left( a\right)
\end{equation*}%
and%
\begin{equation*}
g\left( tb+\left( 1-t\right) a\right) \leq tg\left( b\right) +\left(
1-t\right) g\left( a\right) -ct\left( 1-t\right) \left( b-a\right) ^{2}
\end{equation*}%
By using the elementary inequality, $e\leq f$ and $p\leq r$, then $er+fp\leq
ep+fr$ for $e,f,p,r\in 
\mathbb{R}
,$ then we get%
\begin{eqnarray*}
&&f\left( tb+\left( 1-t\right) a\right) \left[ tg\left( b\right) +\left(
1-t\right) g\left( a\right) -ct\left( 1-t\right) \left( b-a\right) ^{2}%
\right] \\
&&+g\left( tb+\left( 1-t\right) a\right) \left[ tf\left( b\right) +\left(
1-t\right) f\left( a\right) \right] \\
&\leq &f\left( tb+\left( 1-t\right) a\right) g\left( tb+\left( 1-t\right)
a\right) \\
&&+\left[ tf\left( b\right) +\left( 1-t\right) f\left( a\right) \right] %
\left[ tg\left( b\right) +\left( 1-t\right) g\left( a\right) -ct\left(
1-t\right) \left( b-a\right) ^{2}\right] .
\end{eqnarray*}%
So, we obtain%
\begin{eqnarray*}
&&tf\left( tb+\left( 1-t\right) a\right) g\left( b\right) +\left( 1-t\right)
f\left( tb+\left( 1-t\right) a\right) g\left( a\right) -ct\left( 1-t\right)
f\left( tb+\left( 1-t\right) a\right) \left( b-a\right) ^{2} \\
&&+tf\left( b\right) g\left( tb+\left( 1-t\right) a\right) +\left(
1-t\right) f\left( a\right) g\left( tb+\left( 1-t\right) a\right) \\
&\leq &f\left( tb+\left( 1-t\right) a\right) g\left( tb+\left( 1-t\right)
a\right) +t^{2}f\left( b\right) g\left( b\right) \\
&&+t\left( 1-t\right) f\left( b\right) g\left( a\right) +t\left( 1-t\right)
f\left( a\right) g\left( b\right) +\left( 1-t\right) ^{2}f\left( a\right)
g\left( a\right) \\
&&-ct^{2}\left( 1-t\right) \left( b-a\right) ^{2}f\left( b\right) -ct\left(
1-t\right) ^{2}\left( b-a\right) ^{2}f\left( a\right) .
\end{eqnarray*}%
By integrating this inequality with respect to $t$ over $\left[ 0,1\right] $
and by using the change of the variable $tb+\left( 1-t\right) a=x,$ $%
(b-a)dt=dx,$ the proof is completed.
\end{proof}


\begin{thebibliography}{9}
\bibitem{1} M. Alomari, M. Darus, S.S. Dragomir and P. Cerone, Ostrowski's
inequalities for functions whose derivatives are $s-$convex in the second
sense, \textit{RGMIA Res. Rep. Coll.}, Volume 12, 2009, Supplement.

\bibitem{2} B.T. Polyak, Existence theorems and convergence of minimizing
sequences in extremum problems with restictions, \textit{Soviet Math. Dokl.}
7 (1966), 72--75.

\bibitem{3} N. Merentes and K. Nikodem, Remarks on strongly convex functions,%
\textit{\ Aequationes Math}. 80 (2010), no. 1-2, 193-199.

\bibitem{4} K. Nikodem and Zs. Pales, Characterizations of inner product
spaces be strongly convex functions, \textit{Banach J. Math. Anal.} 5
(2011), no. 1, 83--87.

\bibitem{5} H. Angulo, J. Gimenez, A. M. Moros and K. Nikodem, On strongly $%
h-$convex functions, \textit{Ann. Funct. Anal.} 2 (2011), no. 2, 85--91.

\bibitem{6} M.Z. Sar\i kaya and H. Yald\i z, On Hermite-Hadamard type
inequalities for strongly-convex functions, \textit{ONLINE:
http://arxiv.org/pdf/1203.2281.pdf.}

\bibitem{7} M.Z. Sar\i kaya, On Hermite-Hadamard type inequalities for
strongly $\varphi -$convex functions, \textit{ONLINE:
http://arxiv.org/pdf/1203.5497.pdf.}

\bibitem{8} A. Ostrowski, \"{U}ber die Absolutabweichung einer
differentierbaren Funktion von ihren Integralmittelwert, \textit{Comment.
Math. Helv}., 10, 226-227, (1938).

\bibitem{9} M.E. \"{O}zdemir, H. Kavurmac\i\ and E. Set, Ostrowski's type
inequalities for $\left( \alpha ,m\right) -$convex functions, \textit{%
KYUNGPOOK J. Math. }, 50 (2010), 371-378.
\end{thebibliography}
\end{document}